\documentclass[11pt,dvips]{article}
\usepackage{latexsym}
\usepackage{amsmath,amsthm,amsfonts,amssymb} 
\usepackage{url}

\usepackage{enumerate}
\usepackage{mdwlist} 
\usepackage{graphicx}
\usepackage[all]{xy}
\usepackage{export}

\newtheorem{thm}{Theorem}[section]

\newtheorem*{thm*}{Theorem}
\newtheorem{dfn}[thm]{Definition} 
\newtheorem*{dfn*}{Definition}

\newtheorem*{cor*}{Corollary}

\newtheorem{prop}[thm]{Proposition} 
\newtheorem*{prop*}{Proposition} 
\newtheorem*{properties*}{Properties} 
 
\newtheorem{lem}[thm]{Lemma} 
\newtheorem*{lem*}{Lemma}

\newtheorem*{claim*}{Claim} 
 
\newtheorem*{fact*}{Fact}

\newtheorem*{qst*}{Question}

\newtheorem*{pb*}{Problem}

\theoremstyle{remark}
\newtheorem*{rem*}{Remark}
\newtheorem{rem}[thm]{Remark}
\newtheorem*{example*}{Example}

\newenvironment{SauveCompteurs}[1]{%
\newcommand{\monparametre}{#1}
\openexport{\monparametre_sauve}
  \Export{thm}\Export{section}\Export{subsection}\Export{subsubsection}
\closeexport}{}

\newenvironment{UtiliseCompteurs}[1]{%
\newcommand{\monparametre}{#1}
\openexport{\monparametre_aux}
  \Export{thm}\Export{section}\Export{subsection}\Export{subsubsection}
\closeexport
\PackageInfo{export}{\MessageBreak
Importations from \monparametre_sauve.xpt\MessageBreak}%
\InputIfFileExists{\monparametre_sauve.xpt}{\relax}{\relax}%
\renewcommand{\label}[1]{}
}{%
\PackageInfo{export}{\MessageBreak
Importations from \monparametre_aux.xpt\MessageBreak}%
\InputIfFileExists{\monparametre_aux.xpt}{\relax}{\relax}}

\makeatletter
\edef\@tempa#1#2{\def#1{\mathaccent\string"\noexpand\accentclass@#2 }}
\@tempa\rond{017}
\makeatother

\newcommand{\es}{\emptyset}
\renewcommand{\phi}{\varphi} 
\newcommand{\m} {^{-1}}

\newcommand {\ra} {\rightarrow}

\newcommand {\xra} {\xrightarrow}

\newcommand{\ol}[1]{\overline{#1}}

\renewcommand{\subsetneq}{\varsubsetneq}

\newcommand{\dunion}{\sqcup}

\newcommand{\ie} {i.~e.\ }

\newcommand {\calb} {{\mathcal {B}}}   
\newcommand {\calc} {{\mathcal {C}}}   
   
\newcommand {\cale} {{\mathcal {E}}}

\newcommand {\calh} {{\mathcal {H}}}

\newcommand {\calo} {{\mathcal {O}}}

\newcommand {\cals} {{\mathcal {S}}}

\newcommand {\calx} {{\mathcal {X}}}

\newcommand {\bbH} {{\mathbb {H}}}

\newcommand {\bbZ} {{\mathbb {Z}}}   

\newcommand{\abs}[1]{\lvert#1\rvert} 

\newcommand{\VPC} {\mathrm{VPC}}
 \newcommand{\ais}{almost invariant subset}
\newcommand{\Seif}{\mathrm{QH}}

\newcommand{\bo}{\partial}
\newcommand{\cobo}{\delta}
\newcommand{\inc}{\subset}
\newcommand {\Z} {{\mathbb {Z}}} 

\begin{document}

\title{Scott and Swarup's regular neighbourhood as a tree of cylinders}
\author{Vincent Guirardel, Gilbert Levitt}

\maketitle

\begin{abstract}   Let $G$ be a finitely presented group.  Scott and Swarup have constructed a canonical splitting  of $G$  which encloses all almost invariant sets over
virtually polycyclic subgroups of a given length.
We give an alternative construction of  this regular neighbourhood,  by showing that it is 
  the tree of cylinders of a JSJ splitting.
\end{abstract}


\section{Introduction}

Scott and Swarup have constructed in \cite{ScSw_regular} a canonical graph of groups decomposition (or splitting) of a finitely
presented group $G$, which {encloses} all almost invariant sets over virtually polycyclic subgroups of a given length $n$
($\VPC_n$ groups), in particular over virtually cyclic subgroups for $n=1$.

\emph{Almost invariant sets} generalize splittings: a splitting is analogous to an \emph{embedded} codimension-one submanifold of
a manifold $M$, while an almost invariant set is analogous to an \emph{immersed} codimension-one submanifold.   

Two splittings are \emph{compatible} if they have a common refinement, in the sense that both can be obtained from   the refinement
by collapsing some edges.
For example, two splittings induced by   disjoint embedded codimension-one submanifolds are compatible.

\emph{Enclosing} is a generalisation of this notion to almost invariant sets:  in the analogy above,
given two codimension-one submanifolds $F_1,F_2$ of $M$ with $F_1$ immersed and $F_2$ embedded, $F_1$ is enclosed in a 
connected component of $M\setminus F_2$ if one can  isotope $F_1$ into this component.

Scott and Swarup's construction is called the \emph{regular neighbourhood}   of all almost invariant
sets over $\VPC_n$ subgroups. This is analogous to the topological regular neighbourhood of a finite union of (non-disjoint)
immersed codimension-one submanifolds: it defines a splitting which encloses the initial submanifolds.

One main virtue of their splitting is the fact that it is canonical: it is invariant under automorphisms of $G$. Because of this,
it is often quite different from usual JSJ splittings, which are unique only up to  deformation:
the canonical object is the JSJ deformation space \cite{For_uniqueness,GL3}.

 The main reason for this rigidity is that the regular neighbourhood is defined in terms of enclosing.
Enclosing, like  compatibility of splittings, is more rigid than  \emph{domination}, which is the basis for
usual JSJ theory. For instance, any
 two splittings in Culler-Vogtmann's outer space dominate each other, but they are compatible if and only if they lie in a common simplex. 

  On the other hand, we have described a general construction producing a canonical splitting $T_c$ from a   canonical deformation space: 
the \emph{tree of cylinders} \cite{GL4}. 
It also  enjoys strong compatibility properties.
In the present paper we show that   the splitting constructed by Scott and Swarup is a subdivision of the tree of cylinders of the usual JSJ deformation space.

More precisely, let $T_J$ be the Bass-Serre tree of a JSJ splitting of $G$ over $\VPC_n$ groups,  as constructed for instance in
\cite{DuSa_JSJ}.
 To construct the tree of cylinders, 
say that two edges are in the same cylinder if their stabilizers are commensurable. Cylinders are
subtrees, and  the tree  $T_c$ dual to the covering of $T_J$ by cylinders is \emph{the tree of cylinders} of $T_J$
(see \cite{GL4}, or Subsection \ref{sec_constr} below).

\begin{UtiliseCompteurs}{thm_SS}
  \begin{thm}\label{thm_SS}
    Let $G$ be a finitely presented group, and $n\ge1$.
  Assume that $G$ does not split
    over a $\VPC_{n-1}$ subgroup, and that $G$ is not $\VPC_{n+1}$.  Let $T_J$ be a JSJ tree of $G$ over $\VPC_n$ subgroups, and let $T_c$ be
    its tree of cylinders for the commensurability relation.

    Then  the Bass-Serre tree of Scott and Swarup's regular neighbourhood of all almost
    invariant subsets over $\VPC_n$ subgroups  is  equivariantly isomorphic to a subdivision of $T_c$.
  \end{thm}
\end{UtiliseCompteurs}

In particular, this gives a new proof of the fact that this regular neighbourhood is a   tree.
Deriving the regular neighbourhood from a JSJ splitting, instead of building it from an abstract betweenness relation, seems to
greatly simplify the construction; in particular this completely avoids the notion of good or good-enough position for \ais{s}.
\\

There are two ingredients in our approach, to be found in Sections \ref{sec_reg} and \ref{sec_sco} respectively (Section
\ref{sec_prel} recalls basic material about trees of cylinders, almost invariant sets, cross-connected components, regular
neighborhoods).
 
 The first ingredient is a general fact about almost invariant sets  that are \emph{based} on a given tree $T$. 
Consider any simplicial tree $T$ with an action of $G$. 
Any edge $e$ separates $T$
into two half-trees, and this defines  almost invariant sets $Z_e$ and $Z_e^*$ (see Subsection \ref{sec_bas}). 
The collection $\calb(T)$ of 
almost invariant subsets \emph{based} on $T$ is then defined by taking Boolean combinations of such sets $Z_e$.

 Following Scott-Swarup, one defines cross-connected components of $\calb(T)$ by
using \emph{crossing} of almost invariant sets.
The set of cross-connected components is then endowed with a {betweenness} relation which allows one to constructs 
a bipartite graph $RN(\calb(T))$ associated to $\calb(T)$.
This is the \emph{regular neighborhood} of $\calb(T)$ (see Definition \ref{dfn_RN}).

\begin{UtiliseCompteurs}{thm_RN}
  \begin{thm} 
\label{thm_RN} Let $G$ be a finitely generated group, and $T$ a tree with a minimal action of $G$.
Assume that 
 no two groups commensurable to edge stabilizers  are contained in each other with infinite index.

    Then the regular neighbourhood 
$RN(\calb(T))$ is  equivariantly  isomorphic to a subdivision of $T_c$, the tree of cylinders of $T$ for the commensurability relation; in particular, $RN(\calb(T))$ is a tree.
  \end{thm}
\end{UtiliseCompteurs}

The hypothesis about edge stabilizers holds in particular if all edge stabilizers  of $T$ are $\VPC_n$ for a fixed $n$.

This theorem remains true if one enlarges $\calb(T)$ to $\calb(T)\cup QH(T)$, by including almost invariant sets enclosed by
quadratically hanging vertices of $T$. Geometrically, such a vertex is associated to a fiber bundle over a 2-dimensional orbifold
$\calo$. Any simple closed curve on $\calo$ gives a way to blow up $T$ by creating new edges and therefore new almost invariant
sets. These sets are in $QH(T)$, as well as those associated to immersed curves on $\calo$. Under the same hypotheses as Theorem
\ref{thm_RN}, we show (Theorem \ref{thm_RN2}) that the regular neighbourhood $RN(\calb(T)\cup QH(T))$ also is a subdivision of $T_c$.
\\

The second ingredient, specific to the $\VPC_n$ case, is due to (but not explicitly stated by) Scott and Swarup 
\cite{ScSw_regular}.
We believe it is worth emphasizing this statement, as 
it gives a very useful description   of  almost invariant sets  over $\VPC_n$ subgroups
in terms of a JSJ splitting $T_J$:  
in plain words, it says that 
any  almost invariant set  over a $\VPC_n$ subgroup
is either dual to a curve in a QH subgroup, or is a Boolean combination of 
almost invariant sets dual to half-trees  of $T_J$.

\begin{UtiliseCompteurs}{thm_ai}
  \begin{thm}[{\cite{DuSw_algebraic},\cite[Th. 8.2]{ScSw_regular}}]\label{thm_ai}
    Let $G$ and $T_J$ be as in Theorem \ref{thm_SS}.
    For any \ais\ $X$ over a $\VPC_n$ subgroup, the equivalence class $[X]$ belongs to $ \calb(T_J)\cup\Seif(T_J)$.
  \end{thm}
\end{UtiliseCompteurs}

 This theorem is essentially another point of view on the proof of
Theorem 8.2 in \cite{ScSw_regular} (see \cite{ScSw_erratum_regular} for corrections), and makes a crucial use of algebraic torus
theorems \cite{DuSw_algebraic,DuRo_splittings}. We give a proof in Section \ref{sec_sco}.

Theorem \ref{thm_SS} is a direct consequence of Theorems \ref{thm_ai} and \ref{thm_RN2}.

\section{Preliminaries}\label{sec_prel}

In this paper, $G$ will be a fixed finitely generated group. In Section \ref{sec_sco}  it will have to be finitely
presented.
  
\subsection {Trees}
\label{subsec_trees}

 If $\Gamma$ is a graph, we denote by $V(\Gamma)$ its set of vertices and by $E(\Gamma)$ its set of (closed) non-oriented edges. 

A tree  always means a simplicial tree $T$ on which $G$ acts without inversions.   Given a family $\cale$ of subgroups of
$G$, an
\emph{$\cale$-tree} is a tree whose edge stabilizers belong to $\cale$.  We denote by $G_v$ or $G_e$ the
stabilizer of a vertex
$v$ or an edge $e$.

Given a subtree $A$, we denote by $pr_A$ the projection onto $A$, mapping $x$ to the point of $A$ closest to $x$.
If $A$ and $B$ are disjoint, or intersect in at most one point, then $pr_A(B)$ is a single point, and we define
the \emph{bridge} between $A$ and $B$ as the segment joining $pr_A(B)$ to $pr_B(A)$.

A tree $T$ is \emph{non-trivial} if there is no global fixed point,
\emph{minimal} if there is no proper $G$-invariant subtree.

An element or a subgroup of $G$ is \emph{elliptic} in $T$ if it has a global fixed point. An element   which is not elliptic
is \emph{hyperbolic}. It has an axis   on which it acts as a translation. 
 If $T$ is minimal, then it is the union of all translation axes of elements of $G$.
 In particular, if $Y\subset T$ is a subtree, then any connected
component of $T\setminus Y$ is unbounded.  

A subgroup $A$ consisting only of elliptic elements fixes a
point if it is finitely generated, a point or an end in general. If a finitely generated  subgroup $A$ is not elliptic, there is 
a unique minimal $A$-invariant subtree.

A tree $T$ \emph{dominates} a tree $T'$ if there is an equivariant map $f:T\to T'$. Equivalently, any subgroup which is
elliptic in
$T$ is also elliptic in $T'$. Having the same elliptic subgroups is an equivalence relation on the set   of trees, the
equivalence classes are called \emph{deformation spaces} (see  \cite{For_deformation,GL1} for more details).

\subsection{Trees of cylinders }\label{sec_constr}

Two subgroups $A$ and $B$ of $G$ are \emph{commensurable} if $A\cap B$ has finite index in both $A$ and $B$. 

\begin{dfn} \label{dfn_ce}
We fix  a conjugacy-invariant family $\cale$ of subgroups of $G$ such that:
\begin{itemize}
\item any subgroup $A$ commensurable with some $B\in\cale$ lies in $\cale$;
\item if $A,B\in \cale$ are such that $A\subset B$, then $[B:A]<\infty$.
\end{itemize}
 An $\cale$-tree is a tree whose edge stabilizers belong to $\cale$.
\end{dfn}

For instance, $\cale$ may consist of all subgroups of $G$ which are virtually $\bbZ^n$ for some fixed $n$, or all
subgroups which are virtually polycyclic   of Hirsch length exactly
$n$.

 In \cite{GL4} we have      associated a tree of cylinders $T_c$ to any
$\cale$-tree $T$, as follows.  
Two  (non-oriented)  
edges of $T$ are equivalent if their stabilizers are commensurable. 
 A \emph{cylinder} of $T$ is an equivalence class $Y$. We identify $Y$ with the union of its edges, 
which is a  subtree of
$T$.

Two distinct cylinders meet in at most one point.
One   can then define the tree of cylinders of $T$ as the tree $T_c$  dual to the covering of $T$ by its cylinders, as
in \cite[Definition 4.8]{Gui_limit}. Formally, 
 $T_c$ is the bipartite tree with vertex set $V(T_c)=V_0(T_c)\dunion V_1(T_c)$
defined as follows:
\begin{enumerate}
\item $V_0(T_c)$ is the set of vertices $x$ of $T$ belonging to (at least) two distinct cylinders;
\item $V_1(T_c)$ is the set of cylinders $Y$ of $T$;
\item there is an edge $\varepsilon =(x,Y)$ between $x\in V_0(T_c)$ and $Y\in V_1(T_c)$ if and only if $x$ (viewed as a
vertex of $T$) belongs to $Y$ (viewed as a subtree of $T$).
\end{enumerate}
Alternatively, one can define the  \emph{boundary} $\partial Y$ of a cylinder $Y$ as the set of vertices of $Y$ 
belonging to another cylinder, and obtain
$T_c$   from $T$ by replacing each cylinder by the cone on its
boundary.

All edges of a cylinder $Y$ have commensurable stabilizers, and we denote by $\calc\subset\cale$ the corresponding commensurability class.
We sometimes view $V_1(T_c)$ as a set of commensurability classes.

\subsection{Almost invariant subsets}\label{sec_SS}

Given a subgroup $H\subset G$, consider the action of $H$ on $G$  by left multiplication. A subset $X\subset G$ is
\emph{$H$-finite}  if it is contained in the union of finitely many 
$H$-orbits. 
Two subsets $X,Y$ are \emph{equivalent}  if their  symmetric difference $X+Y$ is $H$-finite. We   denote by $[X]$ the 
equivalence class   of $X$, and by $X^*$ the complement of $X$.

An
\emph{$H$-almost invariant subset} (or an almost invariant subset over $H$)  is a subset $X\subset G$ which  is
invariant under the  (left) action of $H$, and such that, for all
$s\in G$, the  
 right-translate $ Xs$ is equivalent to $X$.
An $H$-almost invariant subset $X$ is \emph{non-trivial} if neither $X$ nor its complement $X^*$ is $H$-finite.
 Given $H<G$, the set of equivalence classes of
$H$-almost invariant subsets is a \emph{Boolean algebra $\calb_H$}  for the usual operations.

If $H$ contains $H'$ with finite index, then any $H$-almost invariant subset $X$ is also $H'$-almost invariant.
Furthermore, two sets  $X,Y$ are equivalent over $H'$ if and only if they are  equivalent over $H$. 
It follows that, given a 
commensurability class
$\calc$ of subgroups of
$G$, the set of equivalence classes of almost invariant subsets over   subgroups in $\calc$ is a \emph{Boolean algebra
$\calb_\calc$.}

Two almost invariant subsets $X$ over $H$, and   $Y$ over $K$, are \emph{equivalent} if their symmetric difference
$X+Y$ is
$H$-finite. By   \cite[Remark 2.9]{ScSw_regular}, this is a symmetric relation: $X+Y$ is $H$-finite if and
only if it is
$K$-finite. If $X$ and $Y$ are non-trivial, equivalence implies that  $H$ and $K$ are commensurable. 

The   algebras 
$\calb_\calc$ are thus disjoint, except for the  (trivial) equivalence classes of $\emptyset$ and $G$ which belong to every
$\calb_\calc$. We denote by $\calb$ the union of the algebras $\calb_\calc$. It is the set of equivalence classes of
all almost invariant sets, but it  is not a Boolean algebra  in general.  There is a natural action of $G$ on $\calb$ induced by left translation  (or conjugation).

\subsection{Cross-connected components and regular neighbourhoods  \cite{ScSw_regular}}
\label{sec_ccc}

Let $X$ be an $H$-almost invariant subset, and $Y$ a $K$-almost invariant subset. 
 One says that 
 $X$ \emph{crosses} $Y$, or the pair $\{X,X^*\}$ crosses $\{Y,Y^*\}$, if none of the four sets $X^{(*)}\cap Y^{(*)}$ is
$H$-finite 
(the notation  $X^{(*)}\cap Y^{(*)}$ is a shortcut to denote the four possible intersections $X\cap Y$, $X^*\cap Y$, $X\cap Y^*$, and $X^*\cap Y^*$). 
By \cite{Sco_symmetry},  this is a symmetric relation. 
Note that $X $ and $Y$ do not cross if they are equivalent, and that crossing depends only on the equivalence classes of $X$ and $Y$.
Following \cite{ScSw_regular}, we will say that $X^{(*)}\cap Y^{(*)}$ is \emph{small} if it is $H$-finite (or
equivalently $K$-finite). 

Now
let $\calx$ be a subset of  $\calb$.  
Let
$\ol\calx$ be the set of non-trivial unordered pairs $\{[X],[X^*]\}$ for $[X]\in\calx$. 
A \emph{cross-connected component (CCC) of $\calx$}  is an
equivalence class $C$ for the equivalence relation generated  on $\ol\calx$  
 by crossing. We often say that $X$, rather than $\{[X],[X^*]\}$, belongs to $C$, or represents $C$.
We denote by $\calh$   the set of cross-connected components of $\calx$.

Given three distinct cross-connected components $C_1,C_2,C_3$,
say that $C_2$ is \emph{between} $C_1$ and $C_3$ if there 
are representatives  $X_i$ of $ C_i$  satisfying 
$X_1\subset X_2\subset X_3$.

A \emph{star}   is a   subset   $\Sigma \subset\calh$ containing   at least two elements, and maximal for the
following property:   given $C,C'\in \Sigma $, no $C''\in\calh$ is between $C$ and $C'$. We denote by $\cals$ the set of stars.

\begin{dfn}\label{dfn_RN}  Let $\calx\inc \calb$ be a collection of almost invariant sets. Its
  \emph {regular neighbourhood} $RN(\calx)$   is the bipartite  graph
whose vertex set is  $\calh\sqcup\cals$ (a vertex is either a
   cross-connected component or a   star),
and whose edges are pairs  $(C,\Sigma )\in\calh\times\cals $ with $C\in \Sigma $.  If $\calx$ is $G$-invariant, then $G$ acts on 
$RN(\calx)$ .
\end{dfn}

This definition is motivated by the following remark,  whose proof we leave to the reader.

\begin{rem} \label{rem_eto}

Let $T$ be any simplicial tree. Suppose that $\calh\subset T$ meets any closed edge in  a nonempty
finite set.  Define betweenness in $\calh$   by $C_2\in [C_1,C_3]\subset T$.
Then the bipartite graph defined as above is isomorphic to a subdivision of $T$.
 \end{rem}

The fact that, in the situation of Scott-Swarup, $RN(\calx)$  is   a tree is one of the main results of \cite{ScSw_regular}.
We will reprove this fact by identifying  $RN(\calx)$ with a subdivision of the tree of cylinders.

\section{{Regular neighbourhoods as trees of cylinders } }
\label{sec_reg}

In this section we fix a family $\cale$ as in Definition \ref{dfn_ce}: it is stable under commensurability, and a group of $\cale$
cannot contain another with infinite index.  Let $T$ be an $\cale$-tree.

In the first subsection we define  the set $\calb(T)$ of almost invariant sets based on $T$,
and we state the main result (Theorem \ref{thm_RN}): up to subdivision,
the regular neighbourhood $RN(\calb(T))$ of $\calb(T)$ is the tree of cylinders $T_c$. 
 In Subsection \ref{sec_for}, we represent elements
of $\calb(T)$ by special subforests of $T$. We then study the cross-connected components of  $\calb(T)$,  
and we prove Theorem \ref{thm_RN} in Subsection \ref{sec_pf} by constructing a map $\Phi$ from the set of cross-connected components to $T_c$. In Subsection
\ref{sec_QH} we generalize Theorem \ref{thm_RN} by including almost invariant sets enclosed by quadratically hanging vertices of
$T$ (see Theorem \ref{thm_RN2}).

\subsection{Almost invariant sets based  on  a tree}
\label{sec_bas}

We fix a  basepoint $v_0\in V(T)$.
If $e$ is an edge of $T$, we denote by $\rond e$ the open edge.  
Let  $T_e,T_e^*$  be the connected components of  
$T\setminus  \rond  e $.
The set of $g\in G$ such that
$gv_0\in T_e $ (resp.\ $gv_0\in T_e^*$) is an almost invariant set $Z_e$ (resp. $Z_e^*$) over $G_e$. Up to equivalence, it is
independent of $v_0$. When we need to distinguish between $Z_e$ and $Z_e^*$, we orient $e$ and we declare that the terminal vertex of $e$ belongs to $T_e$.

Now consider a cylinder $Y\subset T$ and  the corresponding commensurability class $\calc$.
Any Boolean combination of $Z_e$'s for $e\in E(Y)$ is an almost invariant set over some subgroup $H\in \calc$.

\begin{dfn}\label{dfn_bool}
  Given a cylinder $Y$, associated to a commensurability class $\calc$, the \emph{Boolean algebra of almost invariant subsets based on $Y$} 
is the subalgebra   $\calb_\calc(T)$ of $\calb_\calc$ generated by the classes $[Z_e]$, for $e\in E(Y)$. 

The set of \emph{almost invariant subsets based on $T$} is the union
 $\calb(T) =\bigcup_{\calc} \calb_\calc(T)$, a subset of 
$\calb =\bigcup_\calc \calb_\calc $ (just like $\calb$, it is a union of Boolean algebras  but   not itself a Boolean algebra).
\end{dfn}

  \begin{prop} \label{prop_bool} Let $T$ and $T'$ be  minimal $\cale$-trees. Then $\calb(T)=\calb(T')$
    if and only if $T$ and $T'$ belong to the same deformation space.

More precisely,
 $T$ dominates $T'$ if and only if $\calb(T')\subset \calb(T)$.   
   \end{prop}

\begin{proof}
Assume that $T$ dominates
$T'$. After subdividing $T$ (this does not change $\calb(T)$), we may assume that there is an equivariant map
  $f:T\ra T'$ sending every edge to a vertex or an edge. We claim that, given $e'\in E(T')$, there are only finitely many 
edges $e_i\in E(T)$ such that
$f(e_i)=e'$.  To see this, we may restrict to a $G$-orbit of edges of $T$, since there are finitely many such
orbits. If $e$ and $ge$ both map onto $e'$, then $g\in G_{e'}$.
Because of the hypotheses on $\cale$, the stabilizer $G_e$ is contained in $G_{e'}$ \emph{with finite index}. The claim follows.

Choose basepoints $v\in T$ and $v'=f(v)\in T'$. Then $Z_{e'}$ (defined using $v'$) is a Boolean combination of the sets
$Z_{e_i}$ (defined using $v$), so $\calb(T')\subset \calb(T)$.

Conversely,  assume $\calb(T')\subset\calb(T)$. 
Let $K\subset G$ be a subgroup elliptic in $T$. We show that it is also elliptic in $T'$. 

If not, we can find 
an edge $e'=[v',w']\subset T'$, and sequences $g_n\in G$ and $k_n\in K$, such that the
sequences $g_nv'$ and $g_nk_nv'$  have no  bounded subsequence, and $e'\subset [g_nv',g_nk_nv']$ for all $n$ (if $K$ contains a hyperbolic
element $k$, we choose $e'$ on its axis, and we define $g_n=k^{-n}$, $k_n=k^{2n}$; if $K $ fixes an end $\omega $, we
 want $g_n\m e'\subset [v',k_nv']$ so
we choose $e'$ and $g_n$ such that all edges $g_n^{-1}e'$ are contained on a  ray $\rho $ going out to $\omega $, and then we choose  $k_n$).
Defining $Z_{e'}$ using the vertex $v'$ and a suitable orientation of $e'$, we have $g_n\in Z_{e'}$ and $g_nk_n\notin
Z_{e'}$. 

Using a vertex of $T$ fixed by $K$ to define the almost invariant sets $Z_e$, we see that any element of $\calb(T)$ is
represented by an almost invariant set $X$ satisfying $XK=X$. In particular, since $\calb(T')\subset\calb(T)$, there
exist finite sets $F_1,F_2$ such that  $Z=(Z_{e'}\setminus G_{e'}F_1)\cup G_{e'}F_2$ is $K$-invariant on the right. For every
$n$ one has  $g_nk_n\in  G_{e'}F_2$  (if $g_n,g_nk_n\in Z$) or   $g_n\in G_{e'}F_1$ (if not), so one of the sequences    $g_nk_nv'$ or $g_nv'$  
 has a bounded subsequence  
(because $G_{e'}$ is elliptic), a contradiction.
\end{proof}

\begin{rem*}
The only fact used in the proof is    that 
no edge stabilizer of $T$ has infinite index in an edge stabilizer of $T'$.
\end{rem*}

We can now state:

\begin{SauveCompteurs}{thm_RN}
\begin{thm}\label{thm_RN}
Let $T$ be a minimal $\cale$-tree, with $\cale$ as in Definition \ref{dfn_ce}, and   $T_c$   its tree of cylinders for the commensurability relation. 
Let $\calx=\calb(T)$ be the set of almost invariant subsets based on $T$.

Then $RN(\calx)$ is  equivariantly isomorphic to a subdivision of $T_c$. \end{thm}
 \end{SauveCompteurs}

Note that $RN(\calx)$ and  $T_c$ only depend on the deformation space of $T$ (Proposition \ref {prop_bool}, and \cite[Theorem 1]{GL4}).  

 To prove the version of Theorem \ref{thm_RN} stated in the introduction, one takes $\cale$ to be the family of subgroups commensurable to an edge stabilizer of $T$. 

 The theorem will be proved in the next three subsections. 
We  always fix a base vertex $v_0\in T$.

\subsection {Special forests}
\label{sec_for}
 
Let $S$, $S'$  be   subsets of $V(T)$.    
We say that $S$ and $S'$ are \emph{equivalent} if their symmetric difference is finite, that $S$   is trivial if it is
equivalent to $\emptyset$ or $V(T)$. 

The
\emph {coboundary $\cobo S$} is the set of edges having   one endpoint in $S$  and one in $S^*$ (the complement of $S$ in $V(T)$).
We shall be interested 
in sets $S$ with finite coboundary.  Since $\cobo (S \cap S')\subset \cobo S \cup \cobo S'$, they form a  Boolean algebra.

We  also  view such an  $S$ as a \emph{subforest} of $T$, by including all edges whose endpoints are both in $S$; we can then consider the  (connected) \emph{components} of $S$. The set of edges of $T$ is partitioned into edges in $S$, edges in $S^*$, and edges in $\cobo S=\cobo S^*$. 
Note that $S$ is equivalent to the set of endpoints of its edges. In particular, $S$ is finite 
(as a   set of vertices)  if and only if it contains finitely many edges.

We say that
$S$ is a \emph{special forest} based on a cylinder $Y$ if $\cobo S=\{e_1,\dots,
e_n\}$ is finite and contained in $Y$.  Note that $S$, if non-empty, contains at
least one vertex of
$Y$.
Each component of $S$ (viewed as a subforest)  is   a component of $T\setminus\{\rond e_1,\dots,
\rond e_n\}$, and   $S^*$ is the union of the other components of $T\setminus\{\rond e_1,\dots,
\rond e_n\}$.

We define $\calb_Y$ as the  Boolean algebra of equivalence classes of special forests based on $Y$.

Given a special forest $S$ based on $Y$,   we define $  X_S=\{g \mid gv_0\in S\}$. It is an
almost invariant set over $H=\cap_{e\in\cobo S}G_e$, 
a  subgroup of $G$ belonging to  the commensurability class $\calc$ associated to $Y$; we denote its equivalence
class by $  [X_S]$.  
Every  element of  $ \calb(T)$ may be represented in this form. More precisely:
\begin{lem} \label{lem_prelim_ai}
Let $Y$ be a cylinder, associated to a commensurability class $\calc$. The map $S\mapsto 
[X_S]$ induces an isomorphism of Boolean algebras between $\calb_Y$ and $\calb_\calc(T)$.
\end{lem}

\begin{proof} It is easy to check that 
$S\mapsto [X_S]$ is a morphism of Boolean algebras. It is onto because  the set   $T_e$ used to define the
almost invariant set 
$Z_e$ is a  special forest (based on the cylinder containing $e$). There remains to determine the ``kernel'',
namely to show that
$X_S$ is $H$-finite if and only if $S$  is finite  (where $H$ denotes any group in $\calc$).

First suppose that $S$ is finite.  
Then $S$ is contained in $Y$ since it contains any connected component of $T\setminus Y$ which
it intersects.
Since $\cobo S$ is finite, no vertex $x$ of $S$ has infinite valence in $T$.
In particular, for each vertex $x\in S$, the group $G_x$ is commensurable with $H$.
It follows that   $\{g\in G\mid g.v_0=x\}$ is $H$-finite, and $X_S$ is $H$-finite.

If $S$ is infinite,  one of its components   is infinite, and by minimality of $T$
 there exists a hyperbolic element $g\in G$  such that  $g^nv_0\in S$ for all $n\ge 0$. Thus $g^n\in X_S$ for $n\ge0$.
 If $X_S $ is $H$-finite,   one can find   a sequence $n_i$ going to infinity,
and $h_i\in H$, such that $g^{n_i} =h_ig^{n_0} $. Since $H$ is elliptic in $T$,
the sequence  $h_ig^{n_0}v_0$ is bounded, a contradiction.
\end{proof}

\begin{lem} \label{lem_prelim2_ai} 
Let
   $S,S'$ be special forests. 
\begin{enumerate}
\item If $S,S'$ are infinite and based on distinct cylinders,  and if $S\cap S'$ is finite,
then $S\cap S'=\es$.
\item If $X_S$ crosses $X_{S'}$, then $S$ and $S'$ are based on the same cylinder.
\item  $X_S\cap X_{S'}$ is small if and only if $S\cap S'$ is finite.
\end{enumerate}
\end{lem}

\begin{proof}
Assume that $S,S'$ are infinite and based on $Y\neq Y'$, and that $S\cap S'$ is finite.
Let $[u,u']$ be the bridge between $Y$ and $Y'$ (with $u=u'$ if $Y$ and $Y'$ intersect in a point).
Since $u$ and $u'$ lie in more than one cylinder, they have infinite valence in $T$.

Assume first that $u\in S$. Then $S$ contains all   components of $T\setminus\{u\}$, except finitely many of them
(which  intersect $Y$).
In particular, $S$ contains $Y'$. If $S'$ contains $u'$, it contains $u$ by the same argument,
and $S\cap S'$ contains infinitely many edges incident on $u$, a contradiction.
If  $S'$ does not contain  $u'$, it is   contained in $S$,
also a contradiction.

We may therefore assume $u\notin S$  and   $u'\notin S'$.
It follows that $S$ (resp.\ $S'$) is contained in the union of the components of $T\setminus\{u\}$
(resp.\ $T\setminus\{u'\}$) which intersect  $Y$ (resp.\ $Y'$), so $S$ and $S'$ are disjoint. This proves Assertion 1.

Assertion 2 may be viewed as a consequence of \cite[Prop. 13.5]{ScSw_regular}. Here is a direct argument.
Assume that $S$ and $S'$ are based on $Y\neq Y'$, and let $[u,u']$ be as above.
Up to replacing $S$  and $S'$  by their complement, we have  $u\notin S$ and $u'\notin S'$. 
The argument above shows that $S\cap S'=\es$, so $X_S$ does not cross $X_{S'}$.

For Assertion 3, first suppose that $S\cap S'$ is finite.
If, say, $S$   is finite, then $X_S$ is $H$-finite by Lemma \ref{lem_prelim_ai}, so $X_S\cap X_{S'}$ is small.  
Assume therefore that $S$ and $S'$ are infinite.
If  they are based on distinct cylinders, 
then  $X_S\cap X_{S'}=\es$ by Assertion 1.
If they are based on the same cylinder, then $S\cap S'$ is itself a finite special forest,
so $X_S\cap X_{S'}=X_{S\cap S'}$ is small by Lemma \ref{lem_prelim_ai}.
Conversely, if $S\cap S'$ is infinite, one shows that $X_S\cap X_{S'}$ is not $H$-finite as in the proof of Lemma
\ref{lem_prelim_ai}, using $g$ such that $g^nv\in S\cap S'$ for all $n\ge 0$.
\end{proof}

\begin{rem} \label{rem_disj}  If $S,S'$ are infinite and $X_S\cap X_{S'}$ is small, then $S$ and $S'$ are equivalent to
disjoint special forests. 
 This follows from the lemma if they are based on distinct cylinders. If not, one replaces $S'$ by $S'\cap S^*$. 
\end{rem}

\subsection {Peripheral cross-connected components}
\label{sec_per}

Theorem \ref{thm_RN} is trivial if $T$ is a line, so we can assume that each vertex of $T$ has valence at least 3
(we now allow $G$ to act with  inversions). We need to understand cross-connected components. By Assertion 2 of
Lemma \ref{lem_prelim2_ai}, every cross-connected component is based on a cylinder, so we focus on a given $Y$.
We first define
\emph{peripheral} special forests and almost invariant sets.

Recall that $\partial Y$ is the set of vertices of $Y$ which belong to another cylinder.
Suppose $v\in \partial Y$ is a vertex whose valence in $Y$ is finite.
Let $e_1,\dots,e_n$ be the    edges of $Y$ which contain $v$, oriented towards  $v$.
Let    $S_{v,Y}=T_{   e_1}\cap \dots \cap T_{  e_n}$ (recall that $T_e$ denotes the component of $T\setminus  \rond  e $ which contains the terminal point of $e$). It is a subtree satisfying
$S_{v,Y}\cap Y=\{v\}$,  
with coboundary $\cobo S_{v,Y}=\{e_1,\dots,e_n\}$. We say that $S_{v,Y}$, and any special forest equivalent to it, is \emph{peripheral} 
(but  $S_{v,Y}^*$ is not peripheral in general).

We denote by $X_{v,Y}$ the almost invariant set corresponding to $S_{v,Y}$, and we say that
$X$ is \emph{peripheral} if it is equivalent to some $X_{v,Y}$.
Both $S_{v,Y}$ and $S_{v,Y}^*$ are infinite, so  $X_{v,Y}$ is non-trivial by Lemma \ref{lem_prelim_ai}.

We claim that $C_{v,Y}=\{\{[X_{v,Y}],[X_{v,Y}^*]\}\}$ is a complete cross-connected component  of $\calb(T)$, 
called a \emph{peripheral} CCC. Indeed, assume that $X_{v,Y}$ crosses some $X_S$.
Then $S$ is based on $Y$ by Lemma \ref{lem_prelim2_ai}, but since $S_{v,Y}$ contains no edge of $Y$ it  is contained in
$S_X$ or $S_X^*$, which  prevents crossing.

Note that if $C_{v,Y}=C_{v',Y'}$ then $Y=Y'$ (because an $H$-almost invariant subset determines the commensurability class of $H$),
and $v=v'$ except when   $Y$ is a single  edge $vv'$ in which case  $X_{v,Y}=X_{v',Y}^*$.

\begin{lem}
Let $Y $ be  a cylinder.
 There is at most one non-peripheral cross-connected component  $C_Y$ based on $Y$. There is
one if and only if $|\partial Y|\ne2,3$.
\end{lem}

\begin{proof}  We divide the proof into three steps.

$\bullet$
We first claim that, given any infinite connected non-peripheral special forest $S$ based on $Y$, there is an edge
$e\subset S\cap Y$ such that both connected components of $S\setminus\{\rond e\}$ are infinite. 

Assume there is no such $e$.
Then $S\cap Y$ is  
locally finite: given $v\in S$, all but finitely many components of
$S\setminus\{v\}$ are infinite, so infinitely many edges incident on $v$ satisfy the claim  if $v$ has infinite valence in
$S\cap Y$.

 Since $S$ is infinite  and non-peripheral,   $S\cap Y$   is not reduced to a single point. We orient every edge $e$ of $S\cap Y$ so that  $S\cap T_e $ is infinite and $S\cap T_e^* $ is finite.    
If a vertex $v$ of $S\cap Y$
is terminal in every edge of
$S\cap Y$ that contains it, $S$ is peripheral. We may therefore find   an infinite ray   $\rho \subset S\cap Y$ consisting of
positively oriented edges. Since every vertex of $T$ has valence $\ge3$, every vertex of $\rho $ is the projection onto
$\rho $ of an edge of $\cobo S$, contradicting the finiteness of
$\cobo S$. This proves the claim.

$\bullet$
To show that there is at most one non-peripheral cross-connected component, we fix 
  two non-trivial   forests $S,S'$ based on $Y$ and we show that
$X_S$ and $X_{S'}$ are in the same CCC, provided that they do not belong to peripheral
CCC's. We can assume that $X_S\cap X_{S'}$   is small, and by Remark \ref{rem_disj}
that  $S\cap S'$ is  empty. We may also assume that every component of $S $ and
$S'$ is infinite.

Since $S$ is not peripheral, it contains two disjoint infinite special forests $S_1,S_2$ based on $Y$: this is clear if
$S$ has several   components, and follows from the claim otherwise. Construct $S'_1,S'_2$ similarly. Then $X_{S_1}\cup
X_{S'_1}$ crosses both $X_S$ and $X_{S'}$, so $X_S$ and $X_{S'}$ are in the same cross-connected component.

$\bullet$
 Having proved the uniqueness of $C_Y$, we now discuss its existence.
If $|\bo Y|\ge4$, choose $v_1,\dots,v_4\in \partial Y$, and consider   edges $e_1,e_2,e_3$ of $Y$
such that each $v_i$ belongs to a different component  $S_i$ of $T\setminus\{\rond e_1,\rond e_2,
\rond e_3\}$. These components are infinite  because $v_i\in\bo Y$, and $X_{S_1\cup S_2} $ belongs to a non-peripheral CCC. 

If  $\bo Y=\es$, then  $Y=T$ and   existence is clear. 
If $\partial Y$ is  non-empty,   minimality of $T$ implies that $Y$ is 
the convex hull of $\partial Y$ (replacing every cylinder by the convex hull of its boundary yields an invariant subtree). From
this one deduces   that  $ | \bo Y | \ne1$, and every CCC based on $Y$ is peripheral   if
$ |\partial Y| $ equals 2 or 3. There is one  peripheral CCC if $|\partial Y|=2$ (\ie $Y$ is a single edge),  
three if
$|\partial Y|=3$. 
\end{proof}

\begin{rem}\label{rem_sep}
The proof shows that, if $\abs{\bo Y}\ge 4$,
then for all $u\neq v$ in $\partial Y$   the non-peripheral CCC   is represented by a   special forest $S $   such that $u\in S$ and $v\in S^*$.
\end{rem}

\subsection{Proof of Theorem \ref{thm_RN}}
\label{sec_pf}

 From now on  we assume that $T$ has more than one cylinder:  otherwise  there is exactly one cross-connected component, and both $RN(\calx)$ and $T_c$  are
points.

 It will be helpful to distinguish between a cylinder $Y\subset T$ or a point $\eta\in \partial Y$, and the corresponding vertex of $T_c$. We therefore
  denote by $Y_c$ or $\eta_c$ the vertex of $T_c$ corresponding to $Y$ or $\eta$.

  Recall that $\calh$ denotes the set of cross-connected components of $\calx=\calb(T)$.
Consider the map $\Phi:\calh\ra T_c$ 
defined as follows:
\begin{itemize}
\item if  $C=C_Y$ is a non-peripheral CCC, then $\Phi(C)=Y_c\in V_1(T_c)$;
\item if  $C= C_{v,Y}$
 is peripheral, and $\#\partial Y\geq 3$,  then
$\Phi(C)$ is the midpoint of the edge $\varepsilon =(v_c,Y_c)$ of $T_c$;
\item if $\#\partial Y=2$, and $C$ is the peripheral CCC based on $Y$, then $\Phi(C)=Y_c$.
\end{itemize}
 In all cases, the distance between $\Phi(C)$ and $Y_c$   is at most $1/2$. If $C$ is peripheral, $\Phi(C)$ has valence 2 in $T_c$.

Clearly, $\Phi$ is one-to-one. By Remark \ref{rem_eto}, it now suffices to show that the image of $\Phi $ meets every
closed edge, and 
$\Phi$ preserves betweenness:
given $C_1,C_2,C_3\in\calh$, then $C_2$ is between $C_1$ and $C_3$ if and only if
  $\Phi(C_2)\in [\Phi(C_1),\Phi(C_3)]$.

The first fact is clear, because $\Phi (\calh)$ contains all vertices $Y_c \in V_1(T_c)$ with $|\bo Y|\ne3$, and the three
points at distance $1/2$ from $Y_c$ if $|\bo Y|=3$. 
To control betweenness, we need a couple of   technical lemmas.

If $S$ is a non-trivial special forest, we   denote by $[[ S]]$ the cross-connected component represented by the almost invariant set $X_S$.  

Let $Y\subset T$ be a cylinder. We denote by $pr_Y:T\ra Y$ the projection.
If $Y' $ is another cylinder, then $pr_Y(Y')$ is a single point.   This point belongs to two cylinders, 
hence defines a vertex of $V_0(T_c)$ which is at distance 1 from $Y_c$ on the segment of $T_c$ joining $Y_c$ to $Y'_c$.

 Let $Y$ be a cylinder  with $ | \bo Y | \geq 4$.
For each  non-trivial special forest $S'$ which is either based on some $Y'\neq Y$, or based on $Y$ and peripheral,
we  define a point $\eta_{Y}(S')\in Y\inc T$  as follows.
If $S'$ is   based on $Y'\ne Y$, we define $\eta_{Y}(S')$ to be $pr_Y(Y')$. 
If $S'$ is equivalent to some  $S_{v,Y}$, we define $\eta_Y(S')=v$;
note that in this case $\eta_Y(S'^*)$ is not defined.

\begin{lem}\label{lem_between1}
Let $Y$ be a cylinder with $| \bo Y |\geq 4$.
Consider two  non-trivial special forests $S, S'$  with $[[S']]\ne C_Y$
and $[[S]]=C_Y$, and assume $S'\subset S$.

Then $\eta=\eta_Y(S')\in Y$ is defined,
$\eta\in S$, and $S'$ contains an equivalent subforest $S''$ with $S''\subset pr_Y\m(\{\eta\})\inc S$.

Moreover, $\Phi([[S']])$ lies in the connected component of $T_c\setminus \{Y_c\}$ containing $ \eta_c $.
\end{lem}

\begin{proof} 
Let $Y'$ be the cylinder on which $S'$ is based.

If $Y'=Y$,  then $S'^*$ is not peripheral, so $S'$ is peripheral. 
 Thus $\eta$ is   defined, and $S'$ is equivalent to its subforest $S''=S_{Y,\eta}$. 
Then $S''=pr_Y\m(\{\eta\})\subset S$. In this case $\Phi([[S']])$ is the midpoint of the edge $(\eta_c,Y_c)$ of $T_c$. 

Assume that $Y'\neq Y$.  Then $\eta=pr_Y(Y')\in S$: 
otherwise $Y'$ would be disjoint from $S$, hence from $S'$,  a contradiction. 
It follows that  $pr_Y\m(\{\eta\})\inc S$. If $\eta\in S'$, then $S'$ contains the complement of 
 $pr_Y\m(\{\eta\})$, so  $S=T$, a contradiction. 
Thus $\eta\notin S'$ and therefore $S'\subset pr_Y\m(\{\eta\})$. 
The ``moreover''   is clear in this case since $\eta_c$ is between $Y_c$ and $Y'_c$,   and $\Phi([[S']])$ is at distance $\le1/2$ from $Y'_c$.
\end{proof}

\begin{lem}\label{lem_between2}
  Let $S=S_{Y,u}$ be peripheral, and let $S'$  be a non-trivial special forest with $[[S']]\ne  [[S]]$.
 Recall that $u_c$ is the vertex of $T_c$ associated to $u$.
\begin{enumerate}
\item If $S'\subset S$, then $\Phi([[S']])$ belongs to the component of  $T_c\setminus\{\Phi([[S]])\}$ which    
  contains  $u _c$. 
\item If $S\subset S'$, then   $\Phi([[S']])$ belongs to the   component of  $T_c\setminus\{\Phi([[S]])\}$ which     does not contain  $u _c$.  
\end{enumerate}
\end{lem}

\begin{proof}
  If $S'\subset S$, then $S'$ is based on some $Y'\neq Y$. 
Since $S'\subset S=pr_{Y}\m(\{ u\})$, we have $Y'\inc pr_{Y}\m(\{ u\})$ and $u_c$ is between $Y_c$ and $Y'_c$ in $T_c$. 
The result follows since $\Phi([[S]])$ is $1/2$-close to $Y_c$ and $\Phi([[S']])$ is $1/2$-close to $Y'_c$.

If $S\subset S'$ and $Y\neq Y'$, we have   $pr_Y(Y')\neq u$ because $S'\ne T$, and the lemma follows.
If $Y=Y'$, the lemma is immediate.
\end{proof}

We  can now show that $\Phi$ preserves betweenness.
Consider three distinct cross-connected components $C_1,C_2,C_3\in \calh$.  Let $Y_2$ be the cylinder on which $C_2$ is based. Note that $| \bo Y_2 |\geq 4$ if $C_2$ is non-peripheral.

First assume that    $C_2$ is between $C_1$ and $C_3$.
By definition, there exist \ais{s} $X_i$  representing $C_i$ such that $X_1\subset X_2\subset X_3$. 
By Lemma \ref{lem_prelim_ai}, one can find special forests $S_i$  with $[X_{S_i}]=[X_i]$. 
By Remark \ref{rem_disj}, since the $C_i$'s are   distinct, 
one can assume $S_1\subset S_2\subset S_3$
(if necessary, replace $S_2$ by $S_2\cap S_3$, then $S_1$ by  $S_1\cap S_2\cap S_3$).

 If  $S_2$ is peripheral, then $\Phi(C_1)$ and $\Phi(C_3)$ are in distinct components
of $T_c\setminus\{\Phi(C_2)\}$ by Lemma \ref{lem_between2}, 
so $\Phi(C_2)\in [\Phi(C_1),\Phi(C_3)]$.
If $S_2^*$ is peripheral,  we apply the same  argument using $S_3^*\subset S_2^*\subset S_1^*$.

Assume   therefore that $C_2$ is non-peripheral.  By Lemma \ref{lem_between1}, the points 
   $\eta_1=\eta_{Y_2}(S_1)$ and $\eta_3=\eta_{Y_2}(S_3^*)$
 are   defined, and $\eta_1\in S_2$ and $\eta_3\in S_2^*$. 
In particular, $\eta_1\neq \eta_3$. By the ``moreover'' we get
 $\Phi(C_2)\in [\Phi(C_1),\Phi(C_3)]$ since   $\Phi(C_2)=( Y_2)_c$.

Now assume that $C_2$ is not between $C_1$ and $C_3$, and choose $S_i$ with $[[S_i]]=C_i$.
By Remark  \ref{rem_disj}, 
we may assume that for each $i\in\{1,3\}$ some inclusion $S_i^{(*)}\subset S_2^{(*)}$ holds. 
  Since $C_2$ is not between $C_1$ and $C_3$,
we may assume after  changing $S_i$ to $S_i^*$ if needed  that $S_1\subset S_2$ and $S_3\subset S_2$.

If $S_2$ or $S_2^*$ is peripheral, Lemma \ref{lem_between2} implies that $\Phi(C_1)$ and $\Phi(C_3)$
lie in the same connected component  of $T_c\setminus\{\Phi(C_2)\}$, so $\Phi(C_2)$ is not between $\Phi(C_1)$ and $\Phi(C_3)$.

Assume therefore  that $C_2$ is non-peripheral. 
By  Lemma \ref{lem_between1},  the points 
 $\eta_1=\eta_{Y_2}(S_1)$ and $\eta_3=\eta_{Y_2}(S_3)$ are defined, and 
we may assume $S_i\subset pr_{Y_2}\m(\{\eta_i\})$.
If $\eta_1=\eta_3$, then $\Phi(C_2)$ does not lie between $\Phi(C_1)$ and $\Phi(C_3)$ by the ``moreover'' of Lemma \ref{lem_between1}.
If $\eta_1\neq \eta_3$, consider $\tilde S_2$ with $[[\tilde  S_2]]= C_2$ such that $\eta_1\in \tilde S_2$ and $\eta_3\in \tilde S_2^*$
 (it exists by Remark \ref{rem_sep}). Then $S_1\subset pr_{Y_2}\m(\eta_1)\subset \tilde S_2$
and  $S_3\subset pr_{Y_2}\m(\eta_3)\subset \tilde S_2^*$
so $C_2$ lies between $C_1$ and $C_3$, a contradiction.

This ends the proof of Theorem \ref{thm_RN}.

\subsection{Quadratically hanging vertices}
\label{sec_QH}

We say that a vertex stabilizer $G_v$  of $T$ is a \emph{QH-subgroup} if there is an exact sequence  $1\ra F\ra G_v\xra{\pi} \Sigma \ra1$, where 
 $\Sigma =\pi_1(\calo)$ 
is  a  hyperbolic $2$-orbifold group  and every incident edge group $G_e$  is peripheral:
 it is contained with finite index in the preimage by $\pi$ 
of a boundary subgroup $B=\pi_1(C)$, with $C$ a boundary component of $\calo$. 
We say that $v$ is a \emph{QH-vertex} of $T$.

We now define almost invariant sets based on $v$. They will be included  in our description of the regular neighbourhood.

We view $\Sigma$ as a convex cocompact Fuchsian group  acting on $\bbH^2$.
Let $\bar H$ be any non-peripheral maximal two-ended subgroup of $\Sigma$ (represented by an immersed curve  or $1$-suborbifold). 
 Let $\gamma$ be the geodesic   invariant by $\bar H$. It separates $\bbH^2$ into two half-spaces   $P^\pm$  
(which may be interchanged by certain elements of $\bar H$). 

Let $\bar H_0$ be the stabilizer of $P^+$, which has index at most $2$ in $\bar H$, 
and $x_0$   a basepoint.
We define an $\bar H_0$-almost invariant set  $\bar X\inc\Sigma$ as the set of $g\in \Sigma$ such that $gx_0\in P^+$ (if $\bar H $ is the fundamental group of a two-sided simple closed curve on $\calo$, 
there is a one-edge splitting of $\Sigma$ over $\bar H$, and $\bar X$ is a $Z_e$  as defined  in   Subsection \ref{sec_bas}).

 The preimage of $\bar X$ in $G_v$ is an almost invariant set $X_v$ over the preimage $H_0$ of $\bar H_0$.
 We extend it to an almost invariant set $X$ of $G$ as follows. Let
 $S'$ be the set of vertices $u\ne v$ of $T$ such that, denoting by
 $e$ the initial edge of the segment $[v,u] $, the geodesic of
 $\bbH^2$ invariant under $G_e\inc G_v$ lies in $P^+$ (note that it lies  in  either   $P^+$ or  $P^-$). Then $X$ is the
 union of $X_v$ with the set of $g\notin G_v$ such that $gv\in S'$.
 
Starting from $\bar H$, we have thus constructed an almost invariant set $X$,
 which is well-defined up to equivalence and complementation (because of the choices of $x_0$ and $P^\pm$).
 We say that   $X$ is a \emph{QH-\ais} based on $v$. We let $\Seif_v(T)$
 be the set of equivalence classes of QH-\ais{s}\ obtained from $v$ as
 above (varying $\bar H$), and $\Seif(T)$ be the union of all
 $\Seif_v(T)$ when $v$ ranges over all QH-vertices of $T$.


\begin{thm}\label{thm_RN2} With $\cale$ and $T$ as in Theorem \ref{thm_RN},   
let $\hat\calx=\calb(T)\cup \Seif(T)$.
Then $RN(\hat\calx)$ is isomorphic to a subdivision of $T_c$.
 \end{thm}
 

 \begin{proof}
   The proof is similar to that of Theorem \ref{thm_RN}.

If $X$ is a QH-\ais\ as constructed above, we call $S=S'\cup\{v\}$   the   \emph{QH-forest} associated to $X$. We say that it is based on $v$.  The coboundary of $S$ is infinite, but all its edges contain $v$. We may therefore view $S$ as a subtree of $T$ (the union of $v$ with certain components of $T\setminus\{v\}$). It is a union of cylinders. We let $S^*=(T\setminus S)\cup\{v\}$, so that $S\cap S^*=\{v\}$. 

 Note that $S$ cannot contain a peripheral special forest $S_{v,Y}$, with $Y$ a cylinder  containing $v$ 
 (this is because the subgroup $\bar H\subset \Sigma$ was chosen non-peripheral).
  
 Conversely, given a QH-forest $S$, one can recover $H_0$, which is the    stabilizer of $S$, and the equivalence class of $X$. In other words, there is a bijection between $QH_v(T)$ and the set of QH-forests based on $v$. 
 We   denote by $X_S$ the   almost invariant set $X$ corresponding to $S$ (it is well-defined up to equivalence). 
 Note that  $X_S\subsetneq \{g\in G \mid gv\in S\}$, and these sets  have the same intersection with $
 G\setminus G_v$. 
 
The following fact is analogous to Lemma \ref{lem_prelim2_ai}.
 
 \begin{lem}\label{lem_for}
 Let $S$ be a QH-forest based on $v$. Let $S'$ be a non-trivial special forest, or a QH-forest based on $v'\ne v$. 
\begin{enumerate}
\item  $X_S\cap X_{S'}$ is small if and only if $S\cap S'=\es$.
\item $X_S$ and $ X_{S'}$ do not cross.
\end{enumerate}
   \end{lem}
   
   \begin{proof}
When $S'$ is a special forest, we use $v$ as a basepoint to define   $X_{S'}=\{g\mid gv\in S'\}$. Beware that     
$X_S$ is properly contained in $\{g\mid gv\in S\}$.

We claim that, if $S'$ is a special forest with $v\notin S'$ and $S\cap S'\neq\es$, then $X_{S'}\subset X_S$. 
Indeed, let $Y'$ be the cylinder on which $S'$ is based.
Since each connected component of $S'$ contains a point in $Y'$, there is a point $w\ne v$ in $S\cap Y'$.
As $S$ is a union of cylinders, $S$ contains $Y'$. All connected components of $S'$ therefore  contain a point of $S$, so are   contained in $S\setminus\{v\}$
  since $v\notin S'$. We deduce  $X_{S'}\subset X_S$.

We now prove assertion 1. If $S\cap S'=\es$, then $X_S\cap X_{S'}=\es$.
We assume     $S\cap S'\neq \es$, and we show that $X_S\cap X_{S'}$ is not small. 
If $S'$ is a QH-forest, then $v\in S'$ or $v'\in S$.
If for instance $v\in S'$, then $X_S\cap X_{S'}$ is not small because it contains $X_S\cap G_v$. 
Now assume that $S'$ is a special forest. If $v\in S'$, the same argument applies,
so assume $v\notin S'$. The claim implies $X_{S'}\subset X_S$, so $X_S\cap X_{S'}$ is not small.

To prove 2, first consider the case where $S'$ is a QH-forest.
Up to changing $S$ and $S'$ to $S^*$ or $S'^*$, one can assume $S\cap S'=\es$
so $X_S$ does not cross $X_{S'}$.
If $S'$ is a special forest, we can assume $v\notin S'$ by changing $S'$ to $S'^*$. By the claim,
$X_S$ does not cross $X_{S'}$.
\end{proof}

The lemma implies that no element of $\Seif(T)$ crosses an element of $\calb(T)$, 
and  elements of $\Seif_v(T)$ do not cross elements of $\Seif_{v'}(T)$ for $v\neq v'$.  
 
Since    $\Seif_{v}(T)$ is a cross-connected component, 
the set $\Hat\calh$ of cross-connected components of $\calb(T)\cup \Seif(T)$ is therefore  the set of cross-connected components of $\calb(T)$,
together with one new cross-connected component $\Seif_v(T)$ for each QH-vertex $v$.

One extends the map $\Phi$ 
defined in the proof of Theorem \ref{thm_RN}
to a map $\Hat \Phi:\Hat\calh\ra T_c$  by sending $\Seif_v(T)$ to $v$
(viewed as a vertex of $V_0(T_c)$ since a QH-vertex belongs to infinitely many cylinders).
We need to prove that $\Hat \Phi$ preserves betweenness.

Lemmas \ref{lem_between1} and \ref{lem_between2} extend immediately to the case where $S'$ is 
a QH-forest: one just needs to define  $\eta_Y(S')=pr_Y(v')$ for $S' $ based on $v'$,
 so that $v'$ plays the role of $Y'$ in the proofs
(in the proof of \ref{lem_between1}, the assertion that  $\eta\notin S'$ should be replaced by the fact that $S'\cap Y$ contains no edge;  this
   holds since otherwise  $S'$ would contain $Y$).
This allows to prove that, if $C_2$ is not a component $\Seif_v(T)$, then $\Phi(C_2)$ is between $\Phi(C_1)$ and $\Phi(C_3)$ 
if and only if $C_2$ lies between $C_1$ and $C_3$.

To deal with the case when $C_2=\Seif_v(T)$, we need a cylinder-valued projection $\eta_v$.
Let $Y$ be a cylinder or a QH-vertex distinct from $v$.
We define $\eta_v(Y)$ as the cylinder of $T$ containing the initial edge of $[v,x]$ for any  $x\in Y$ different from $v$.
Equivalently, $\eta_v(Y)$ is $Y$ if $v\in Y$,   the cylinder containing the initial edge
of the bridge joining $x$ to $Y$ otherwise.

  If $v$ lies in a cylinder $Y^0$, denote by $\eta_v\m(Y^0)$ the union of cylinders $Y$ such that $\eta_v(Y)=Y^0$.
Equivalently, this is the set of points $x\in T$ such that $x=v$ or $[x,v]$ contains an edge of $Y^0$.

As before, we denote by $[[S]]$ the cross-connected component represented by $X_S$.

\begin{lem}\label{lem_between3} 
 Let $S$ be a QH-forest based on $v$. Let $S'$ be a non-trivial special forest, or a QH-forest based on $v'\ne v$. 
Let $Y'$ be the cylinder or QH-vertex on which $S'$ is based, and
 $Y'^0=\eta_v(Y')$.
 
If $S'\inc S$,
then $S'\subset \eta_v\m(Y'^0)\subset S$. 

Moreover,
$\Phi([[S']])$ and $Y'^0_c$ lie in the same component of $T_c\setminus \{\Phi([[S]])\}$.
\end{lem}

We leave the proof of  this lemma to the reader.

Assume now that $S_1\subset S_2 \subset S_3$ with $[[S_i]]=C_i$ and $S_2 $ based on $v$.
For $i=1,3$ let $Y_i^0=\eta_v(Y_i)$.
Then $S_1 \subset \eta_v\m(Y_1^0)\subset S_2$ and $S_3^* \subset \eta_v\m(Y_3^0)\subset S_2^*$.
In particular, $Y_1^0\neq Y_3^0$. Since $ (Y_1^0)_c$ and $ (Y_3^0)_c$ are neighbours of $v_c$,
they  lie in distinct components of $T_c\setminus \{\Phi(C_2)\}$.
By Lemma \ref{lem_between3}, so do $\Phi([[S_1]])$ and $\Phi([[S_3]])$.

Conversely, assume that $C_2$ does not lie between $C_1$ and $C_3$, and consider
$S_1\subset S_2$ and $S_3\subset S_2$ with  $[[S_i]]=C_i$.
For $i=1,3$, let $Y_i^0$ be as above.
If $Y_1^0=Y_3^0$, then $\Phi(C_2)$ is not between $\Phi(C_1)$ and $\Phi(C_3)$ by Lemma \ref{lem_between3}, and we are done. 
If $Y_1^0\neq Y_3^0$, these cylinders  correspond to distinct peripheral subgroups of $G_v$, with  invariant geodesics $\gamma_1\ne \gamma_3$. 
There exists a non-peripheral group $\bar H\inc \Sigma$, as in the beginning of this subsection,  
whose invariant geodesic separates $\gamma_1$ and $\gamma_3$. Let $S'_2$ be the associated QH-forest. Then
  $[[S'_2]]=C_2$ and, up to complementation, $\eta_v\m(Y_1^0)\subset S'_2$ and $\eta_v\m(Y_3^0)\subset S'_2{}^*$.
It follows that $S_1\subset S'_2$ and $S_3^*\subset S'_2{}^*$, so $C_2$ lies between $C_1$ and $C_3$, contradicting our assumptions.
 \end{proof}

\section{The regular neighbourhood of Scott and Swarup }
\label{sec_sco}

A group is $\VPC_n$ if some finite index subgroup is polycyclic of Hirsch length $n$. For instance,   $\VPC_0$ groups are finite groups, $\VPC_1$ groups are virtually cyclic groups, $\VPC_2$ groups are virtually $\Z^2$ (but not all $\VPC_n$ groups are virtually abelian for $n\ge3$).

Fix $n\ge1$. We assume that $G$ is finitely presented and does not split over a
    $\VPC_{n-1}$ subgroup. 
 We also assume that $G$ itself is not $\VPC_{n+1}$.
All trees considered here are assumed to have  
    $\VPC_n$ edge stabilizers.

A   subgroup $H\subset G$ is \emph{universally elliptic} if it is elliptic in every
 tree. A
 tree is universally elliptic if all its edge stabilizers are.

A tree is a \emph{JSJ tree} (over $\VPC_n$ subgroups) if it is universally elliptic, and maximal for
this property: it dominates every universally elliptic tree. JSJ trees   exist (because $G$ is finitely presented) and
  belong to the same deformation space, called
the  JSJ deformation space  (see \cite{GL3}). 

 A vertex stabilizer $G_v$ of a JSJ tree is \emph{flexible} if it is not $\VPC_n$ and is not
universally elliptic. It follows from \cite{DuSa_JSJ} that a flexible vertex stabilizer is a QH-subgroup (as defined in Subsection \ref{sec_QH}): 
 there is an exact sequence $1\to F\ra G_v\ra \Sigma\to 1$, where 
$\Sigma =\pi_1(\calo)$  is the fundamental group of a  hyperbolic $2$-orbifold,  
$F$ is $\VPC_{n-1}$, and every incident edge group $G_e$ is peripheral.  Note that the QH-almost invariant subsets $X$ constructed in Subsection \ref{sec_QH} are over $\VPC_n$ subgroups.

We can now describe the regular neighbourhood of all \ais{s} 
of $G$ over $\VPC_n$ subgroups as a tree of cylinders.

\begin{SauveCompteurs}{thm_SS}
  \begin{thm}\label{thm_SS}
    Let $G$ be a finitely presented group, and $n\ge1$. Assume that $G$ does not split over a
    $\VPC_{n-1}$ subgroup, and that $G$ is not $\VPC_{n+1}$.  
     Let $T$ 
      be a  JSJ tree  over $\VPC_n$ subgroups, and
    let $T_c$ be its tree of cylinders for the commensurability
    relation.

    Then  Scott and Swarup's
    regular neighbourhood of all almost invariant subsets over
    $\VPC_n$ subgroups is  equivariantly isomorphic to a subdivision of $T_c$.
  \end{thm}
\end{SauveCompteurs}

This is  immediate    from Theorem \ref{thm_RN2} and the following result saying  that one can read 
any almost invariant set over a $\VPC_n$  subgroup in a  JSJ tree $T$.

\begin{SauveCompteurs}{thm_ai}
\begin{thm}[{\cite{DuSw_algebraic},\cite[Th. 8.2]{ScSw_regular}}]\label{thm_ai}
  Let $G$ and $T$  be as above.
For any \ais\ $X$ over a $\VPC_n$ subgroup,  
the equivalence class $[X]$   belongs to $ \calb(T)\cup\Seif(T)$.  
\end{thm}
\end{SauveCompteurs}

\begin{proof} 
We essentially follow the proof by Scott and Swarup. 
For definitions, see \cite{ScSw_regular}. All trees considered here have $\VPC_n$ edge stabilizers.

Let $X$ be an \ais\ over a $\VPC_n$ subgroup $H$.  We assume that it is non-trivial.
 We first consider the case where  $X$ crosses strongly some other \ais. Then by
\cite[Proposition 4.11]{DuSw_algebraic}  $H$ is contained as a non-peripheral subgroup in a QH-vertex stabilizer $W$ of some   tree $T'$. 
 When acting on $T$, the group $W$   fixes a   QH-vertex $v\in T$   (see \cite{GL3}).

Note that $H$ is not peripheral in $G_v$, because  it is not peripheral in $W$.
Since $(G,H)$  
only has   $2$ co-ends \cite[Proposition 13.8]{ScSw_regular}, 
there are (up to equivalence) only two 
\ais{s} over subgroups commensurable with $H$  (namely $X$ and $X^*$), and therefore $[X]\in \Seif_v(T)$.

>From now on, we assume that $X$ crosses strongly no other \ais{}
 over a $\VPC_n$ subgroup.  
Then, by \cite{DuRo_splittings} and \cite[Section 3]{DuSw_algebraic}, 
there is a  non-trivial tree $T_0$ with one orbit of edges 
and an edge stabilizer   $H_0$ commensurable with $H$.

Since $X$ crosses strongly no other almost invariant set, $H$ and $H_0$ are universally elliptic
(see \cite[Lemme 11.3]{Gui_coeur}). 
In particular, $T$ dominates $T_0$. It follows that there is an edge of $T$ with stabilizer contained in $H_0$ (necessarily with finite index).  
This edge is contained in a cylinder $Y $ associated to the commensurability class of $H$.

 The main case is when $T$ has no edge $e$ such that   $Z_e$  crosses $X$  (see Subsection \ref{sec_bas} for the definition of $Z_e$).
The following lemma implies that $X$ is enclosed in some vertex $v$ of $T$.

\begin{lem} 
Let $G$ be finitely generated. 
   Let $X\inc G$ be a non-trivial almost invariant set over a finitely generated subgroup $H $. Let
  $T$ be a  tree with an action of $G$. If $X$ crosses no $Z_e$, then $X$ is enclosed in some vertex $v\in T$.
\end{lem}

\begin{proof} 
The argument follows a part of the proof of Proposition 5.7 in \cite{ScSw_regular,ScSw_erratum_regular}.

Given two  almost  invariant subsets, we use the notation $X\geq Y$ when $Y\cap X^*$ is small.
The non-crossing hypothesis says that  each edge $e$ of $T$ may be oriented  so that
$Z_e \geq X$ or $Z_e \geq X^*$.
If one can choose both orientations for  some $e$, then $X$ is equivalent to $Z_e$, so $X$ is enclosed
in both endpoints of $e$ and we are done.

We orient each edge of $T$ in this manner.
We color the edge blue or red, according to whether
$Z_e \geq X$ or $Z_e \geq X^*$.  No  edge can  have both colors.
If $e$ is an oriented edge, and if $e'$ lies in  $T_{e}^*$, then $e'$ is oriented towards $e$,
so that $Z_{e}\subset Z_{e'}$, and $e'$ has the same color as $e$.
In particular, given a vertex $v$, either all edges containing $v$ are oriented towards $v$,
or there exists exactly one edge containing $v$ and oriented away from $v$, and all edges containing $v$
have the same color.

If $v$ is as in the first case, $X$ is enclosed in $v$ by definition.
If there is no such $v$, then all edges have the same color and are oriented
towards an end of $T$.
By Lemma 2.31 of \cite{ScSw_regular}, $G$ is contained in the $R$-neighbourhood of $X$
for some $R>0$, so $X$ is trivial, a contradiction.
\end{proof}

 Let $v$ be a vertex of $T$ enclosing $X$. In particular $H\subset G_v$.
The set $X_v=X\cap G_v$ is an $H$-\ais\ of $G_v$ (note that $G_v$ is finitely generated). 
By \cite[Lemma 4.14]{ScSw_regular}, 
there is a subtree $S\subset T $ containing $v$,  with $ S\setminus\{v\}$ a union of components of   $T\setminus\{v\}$, such that 
$X$ is equivalent to $X_v\cup \{g\mid g.v\in S\setminus\{v\}\}$.

\begin{lem} The $H$-\ais\ $X_v$ of $G_v$ is trivial.
\end{lem}

\begin{proof}
Otherwise, by \cite{DuRo_splittings,DuSw_algebraic}, 
 there is a $G_v$-tree $T_1$ with one orbit of edges and an edge stabilizer   $H_1$ commensurable with $H$,
and   an edge $e_1\inc T_1$, such that
$Z_{e_1}$ lies (up to equivalence) 
in the Boolean algebra generated by the  orbit of $X_v$ under the commensurator of $H$ in $G_v$.

Note that $G_e$ is elliptic in $T_1$ for each edge $e$ of $T$ incident to $v$:
  by symmetry of strong crossing (\cite[Proposition 13.3]{ScSw_regular}), 
$G_e$ does not cross strongly any translate of $X$, and thus does not cross strongly $Z_{e_1}$, so
$G_e$ is elliptic in $T_1$ (\cite[Lemme 11.3]{Gui_coeur}). 
This ellipticity allows us to refine $T$ by creating new edges with stabilizer conjugate to $H_1$.
Since $H_1$ is universally elliptic, this contradicts the maximality of  the JSJ tree $T$. 
\end{proof}

After  replacing
$X$ by an equivalent \ais{} or its complement, and possibly changing $S$  to $(T\setminus S)\cup \{v\}$, 
we can assume that  $X =\{g\mid g.v\notin S  \}$.
  Recall that $Y$ is the cylinder defined by the commensurability class of $H$.

\begin{lem}\label{cl_Sigma}
 The   coboundary $\cobo S$, consisting  of edges $vw$ with $w\notin S$,  is a finite set of edges of $Y$.  
\end{lem}

This implies that $[X]\in\calb(T)$, ending the proof when $X$ crosses no $Z_e$.

\begin{proof}[Proof of Lemma \ref{cl_Sigma}]
Let $E$ be the set of edges of $\cobo S$, oriented so that $X=\dunion_{e\in E} Z_e$ (we use $v$ as a basepoint to define $Z_e$).
Let $A$ be a finite generating system of $G$ such that, for all $a\in A$, the open segment $(av,v)$ does not meet the orbit of $v$.
One can construct such a generating sytem from any finite generating system   by iterating the following operation: replace $\{a\}$ by the pair $\{g,g\m a\}$
if   $(av,v)$ contains some $g.v$.

Let $\Gamma$ be the Cayley graph of $(G,A)$.
For any subset $Z\subset G$, denote by $\cobo Z$ 
the set of edges of $\Gamma$ having one endpoint in $Z$ and the other endpoint in $G\setminus Z$.
By our choice of $A$, no edge joins a vertex of $Z_e$ to a vertex of $Z_{e'}$ for $e\neq e'$.
It follows that $\cobo X=\dunion_{e\in E} \cobo Z_{e}$.  

Since $\cobo X$ is $H$-finite, the set $\cobo Z_{e}$ is $H$-finite for each $e\in E$, and $E$ is contained in a finite union of 
  $H$-orbits.  Let $e\in E$. 
Since $\cobo Z_e$ is $G_e$-invariant and $H$-finite, $G_e\cap H$ has finite index in $G_e$.
Since $G_e$ and $H$ are both $\VPC_n$, they are commensurable, so the $H$-orbit of $e$ is finite.
It follows that $E\subset Y$ and that $E$ is finite.
\end{proof}

We now turn to the case when $X$ crosses some $Z_e$'s.
For each $e\in E(T)$, the intersection number $i(Z_e,X)$ is finite \cite{Sco_symmetry}, 
which means that there are only finitely many 
edges $e'$ in the orbit of $e$ such that $Z_{e'}$ crosses $X$.
Since $T/G$ is finite, let $e_1,e_1\m, e_2, e_2\m,\dots,e_n, e_n\m$ be the finite set of oriented edges $e$ such that $Z_e$ crosses $X$
(we denote by $e\mapsto e\m$ the orientation-reversing involution).
Note that $e_i\subset Y$ by \cite[Proposition 13.5]{ScSw_regular}.
Now $X$ is a finite union of sets of the form $X'=X\cap Z_{e_1^{\pm 1}} \cap \dots\cap Z_{e_n^{\pm 1}}$.
Since $X'$ does not cross any $Z_e$, its equivalence class lies in $\calb(T)$ by the argument above and so does $[X]$.
\end{proof}


 \def\cprime{$'$}

\begin{flushleft}
Vincent Guirardel\\
Institut de Math\'ematiques de Toulouse\\
Universit\'e de Toulouse et CNRS (UMR 5219)\\
118 route de Narbonne\\
F-31062 Toulouse cedex 9\\
France.\\
\emph{e-mail:}\texttt{guirardel@math.ups-tlse.fr}\\[8mm]

Gilbert Levitt\\
Laboratoire de Math\'ematiques Nicolas Oresme\\
Universit\'e de Caen et CNRS (UMR 6139)\\
BP 5186\\
F-14032 Caen Cedex\\
France\\
\emph{e-mail:}\texttt{levitt@math.unicaen.fr}\\
\end{flushleft}

\end{document}